\documentclass[12pt]{article}
\usepackage[left=2.8cm,right=2.8cm,top=2.2cm,bottom=2.2cm]{geometry}
\usepackage{amsfonts}
\usepackage{titlesec}
\usepackage{cite,color,xcolor}
\usepackage{amsmath}
\usepackage{amssymb}
\usepackage{times}
\usepackage{bm}
\usepackage[colorlinks,citecolor=blue,urlcolor=blue]{hyperref}
\usepackage[utf8]{inputenc}
\usepackage{mathtools}
\usepackage{amsthm}
\usepackage{setspace}
\usepackage[numbers]{natbib}
\usepackage{cases}
\usepackage{fancyhdr}
\usepackage[pagewise]{lineno}

\newtheorem{theorem}{Theorem}[section]
\newtheorem{corollary}{Corollary}[section]
\newtheorem{lemma}{Lemma}[section]

\newtheorem{definition}{Definition}[section]

\usepackage{txfonts}
\newcommand{\bal}{\begin{align}}
\newcommand{\bbal}{\begin{align*}}
\newcommand{\beq}{\begin{equation}}
\newcommand{\eeq}{\end{equation}}
\newcommand{\bca}{\begin{cases}}
\newcommand{\eca}{\end{cases}}
\def\div{\mathord{{\rm div}}}

\newcommand{\fr}{\frac}

\newcommand{\cd}{\cdot}
\newcommand{\ep}{\epsilon}
\newcommand{\dd}{\mathrm{d}}

\newcommand{\R}{\mathbb{R}}

\newcommand{\f}{\left}
\newcommand{\g}{\right}

\begin{document}
\bibliographystyle{plain}
\title{A remark on the vanishing diffusivity limit of the Keller-Segel equations in Besov spaces}

\author{Yanghai Yu\footnote{E-mail: yuyanghai214@sina.com(Corresponding author); lf191110@163.com} and Fang Liu\\
\small   School of Mathematics and Statistics, Anhui Normal University, Wuhu 241002, China}

\date{\today}
\maketitle\noindent{\hrulefill}

{\bf Abstract:} It is shown in \cite[J. Differ. Equ., (2022)]{22jde} that given initial data $u_0\in B^{s}_{p,r}$ and for some $T>0$, the solutions of the parabolic-type Keller-Segel equations converge strongly in $L^\infty_TB^{s}_{p,r}$ to the hyperbolic Keller-Segel equations as the diffusivity parameter $\epsilon$ tends to zero. In this paper, we furthermore prove this solution maps do not converge uniformly with respect to the initial data $u_0$ as $\epsilon\to0$ in the same topology of Besov spaces.


\vskip0mm\noindent{\hrulefill}

\section{Introduction}
For more than a century, biologists have observed that certain species of bacteria are preferred to move
toward higher concentrations of some chemicals, such as minerals, oxygen and organic nutrients. This biased movement, universally referred to as chemotaxis, has been fueling interest of both experimentalists and
theoreticians since it plays a vital role in wide-ranging biology phenomena \cite{hil}. Many diverse disciplines involve chemotaxis models whose
aspects include not only the mechanistic basis and biological foundations but also the modeling of specific systems and the mathematical analysis of the governing nonlinear equations. The most prominent model for this process was derived by Patlak, Keller and Segel \cite{Patlak01, Keller01,Keller02, Keller03},
which takes the form of
\begin{equation}\label{ame}
       \begin{cases}
        \partial_{t}u+ \div(D_{1}(u,S)\nabla u-\chi(u,S)\nabla S)=0,\\
        \tau S_{t}=D_{2} \Delta S+k(u,S),
       \end{cases}
     \end{equation}
here $u(x,t)$ represents the cell density at position $x \in \mathbb{R}^{d}$, time $t > 0$, and $S(x,t)$ is the concentration of a chemical signal. The motility $D_{1}(u, S)$ and the chemotactic sensitivity $\chi(u, S)$
rely on the cell density and on the signal concentration. The term $k(u, S)$ depicts production
and decay or consumption of the signal and $D_{2}$ is the diffusion constant for $S$. The parameter
$\tau$ illustrates that movement of the species and dynamics of the signal have different characteristic time scales. The Keller-Segel model has been applied to many different problems, ranging from bacteria
chemotaxis to cancer growth or the immune response.

Inspired by the convection equation with a small diffusion term as higher order correction
from a kinetic model for chemotaxis, by taking $D_{1}(u,S)=-\epsilon,\, \chi(u,S)=1-u,\, \tau=0,\, D_2=1$ and $k(u,S)=S-u$ in \eqref{ame}, Dolak-Schmeiser \cite{Dolak02} proposed the following parabolic-type Keller-Segel equations with small diffusivity
\begin{equation}\label{wme}
       \begin{cases}
        \partial_{t}u=- \div(u(1-u)\nabla S-\epsilon\nabla u), \\
       -\Delta S=u - S.
       \end{cases}
     \end{equation}
In this paper, we consider the Cauchy problem for the parabolic-type Keller-Segel equation \eqref{wme}
\begin{align}\label{1}\tag{PKS}
       \begin{cases}
        \partial_{t}u-\epsilon\Delta u=-\div\f(u(1-u)\nabla S\g), &\text{in}\quad \mathbb{R}^{+}\times\mathbb{R}^{d}, \\
        S=(1-\Delta)^{-1}u,          &\text{in}\quad    \mathbb{R}^{+}\times\mathbb{R}^{d}, \\
        u(x,0)=u_{0}(x)               , &\text{in}\quad \mathbb{R}^{d}.
       \end{cases}
     \end{align}
When $\epsilon=0$, \eqref{1} reduces to the hyperbolic Keller-Segel equation
     \begin{align}\label{2}\tag{HKS}
       \begin{cases}
        \partial_{t}u=-\div\f(u(1-u)\nabla S\g), &\text{in}\quad \mathbb{R}^{+}\times\mathbb{R}^{d}, \\
        S=(1-\Delta)^{-1}u,          &\text{in}\quad    \mathbb{R}^{+}\times\mathbb{R}^{d}, \\
        u(x,0)=u_{0}(x)               , &\text{in}\quad \mathbb{R}^{d}.
       \end{cases}
     \end{align}
Dolak-Schmeiser \cite{Dolak01} firstly
established the existence and unique of global smooth solution to one dimensional version of \eqref{wme} with suitable
conditions on the initial data. On a time scale characteristic for the convective effects, they also
proved that the corresponding sequence of solutions $u^\epsilon$ converges to the weak entropy solution
$u$ to \eqref{ame} as $\epsilon\to0$. Burger-Difrancesco-Dolak\cite{Burger01} obtained the unique local-in-time solution to \eqref{wme} with the initial data belonging to $L^1(\R^d)\cap L^\infty(\R^d)$. Perthame-Dalibard \cite{Perthame01} proved the existence of an entropy solution to \eqref{ame} by passing to the limit in a sequence of solutions
to the parabolic approximation. Burger-Dolak-Schmeiser \cite{Burger02} studied the asymptotic behavior of solutions of the chemotaxis model \eqref{wme} in multi-dimensional spaces.  Some other results related to \eqref{wme} can be found in \cite{Tello01,Winkler01,Winkler02}.

Zhou-Zhang-Mu \cite{Zhou01} obtained the existence and uniqueness of solution of \eqref{2} in ${{B}}_{p,r}^{s}(\mathbb{R}^{d})$ wtih $1\leq p,r\leq\infty$ and $s>1+{d}/{p}$. Zhang-Mu-Zhou \cite{22jde} proved that \eqref{2} is well-posed in ${{B}}_{p,1}^{1+{d}/{p}}(\mathbb{R}^{d})$ with $1\leq p<\infty$ and is ill-posed in ${{B}}_{2,\infty}^{s}(\mathbb{R})$ with $s>\frac{3}{2}$. Fei-Yu-Fei \cite{Fyf} generalized the above result by proving that \eqref{2} is ill-posed in $B^s_{p,\infty}(\R^d)$ with $d\geq1,1\leq p\leq\infty$ and $s>1+{d}/{p}$.
Formally, as $\epsilon\to0$, the solution of \eqref{1} converges to the solution of \eqref{2}, which has been considered in \cite{Dolak02} for one dimension case. Based on the well-posedness results in \cite{Zhou01,22jde}, Zhang-Mu-Zhou \cite{22jde} proved that given initial data $u_0\in B^{s}_{p,r}$ and for some $T>0$, the solutions of the parabolic-type Keller-Segel equations \eqref{1} converge strongly in $L^\infty_TB^{s}_{p,r}$ to the hyperbolic Keller-Segel equations \eqref{2} as the diffusivity parameter $\epsilon$ tends to zero. In this paper, we furthermore prove this solution maps do not converge uniformly with respect to the initial data $u_0$ as $\epsilon\to0$ in the same topology of Besov spaces.

Before stating our main result, we denote any bounded subset $U_R$ in $B^s_{p,r}(\mathbb{R}^d)$ by
$$U_R:=\left\{\phi\in B^s_{p,r}(\mathbb{R}^d): \|\phi\|_{B^s_{p,r}(\mathbb{R}^d)}\leq R\right\}\quad \text{for}\quad R>0.$$
Let
 \begin{align*}
&u^\epsilon(t,u_0)=\text{\rm the solution map of}\, \eqref{1} \text{\rm with initial data}\, u_0,\\
&\bar{u}(t,u_0)=\text{\rm the solution map of}\, \eqref{2} \text{\rm with initial data}\, u_0.
 \end{align*}
 The main result of the paper is the following theorem:
\begin{theorem}\label{th2} Let $d\geq 1$. Assume that $(s,p,r)$ satisfies
\begin{align}\label{cond}
s>\frac{d}{p}+1,\; (p,r)\in [1,\infty]\times(1,\infty) \quad   \text{or}    \quad s=\frac{d}{p}+1,\; (p,r)\in [1,\infty)\times\{1\}.
\end{align}
 Then there exists a sequence initial data $u_0\in U_R$ such that for a short time $T_0$
$$
\liminf_{\epsilon_n\to 0}\left\|u^{\epsilon_n}(t,u_0)-\bar{u}(t,u_0)\right\|_{L^\infty_{T_0}B^s_{p,r}}\geq c_0,
$$
with some positive constant $c_0$.
\end{theorem}
As a by-product, we have
\begin{corollary}\label{co1} Let $d\geq 1$. Assume that $(s,p,r)$ satisfies \eqref{cond}
For any $u_0\in U_R$, let $\mathbf{S}_{t}^{\epsilon}(u_0)$ be the solutions of \eqref{1} with the initial data $u_0$. Then a family of solutions $\f\{\mathbf{S}_{t}^{\epsilon}(u_0)\g\}_{\epsilon>0}$ to \eqref{1}
\begin{equation*}
\mathbf{S}_t^\epsilon:\begin{cases}
U_R \rightarrow \mathcal{C}([0, T] ; B_{p, r}^{s}),\\
u_0\mapsto \mathbf{S}_t^\epsilon(u_0),
\end{cases}
\end{equation*}
does not converge uniformly with respect to initial data  as $\epsilon\to0$ to the solutions to \eqref{2} in $B^{s}_{p,r}$.
\end{corollary}

\section{Littlewood-Paley Analysis}\label{sec2}
\setcounter{equation}{0}
We will use the following notations throughout this paper.
 For $X$ a Banach space and $I\subset\R$, we denote by $\mathcal{C}(I;X)$ the set of continuous functions on $I$ with values in $X$. Sometimes we will use $\|(f,g)\|_{X}=\|f\|_{X}+\|g\|_{X}$ and $L_T^pX=L^p(0,T;X)$.
The symbol $\mathrm{A}\approx \mathrm{B}$ means that $C_1\mathrm{B}\leq\mathrm{A}\leq C_2\mathrm{B}$ for some uniform positive ``harmless" constants $C_1$ and $C_2$.
 Let us recall that for all $u\in \mathcal{S}'$, the Fourier transform $\widehat{u}$ is defined by
$$
\mathcal{F}u(\xi)=\widehat{u}(\xi)=\int_{\R^d}e^{-\mathrm{i}x\cd \xi}u(x)\dd x \quad\text{for any}\; \xi\in\R^d.
$$
Next, we will review the definition of Littlewood-Paley decomposition and nonhomogeneous Besov space, and then list some useful properties which will be frequently used in the sequel. For more details, the readers can refer to \cite{B}.

Let $\varphi\in C_c^{\infty}(\mathbb{R}^d)$ and $\chi\in C_c^{\infty}(\mathbb{R}^d)$ be  radial positive functions such that
    \begin{align*}
    &\mathrm{supp}\ \varphi\subset\f\{\xi\in \mathbb{R}^d:\frac34\leq|\xi|\leq\frac83\g\},  \quad\mathrm{supp}\ \chi\subset\f\{\xi\in \mathbb{R}^d:|\xi|\leq\frac43\g\},\quad
    \\& \chi(\xi)+ \sum_{j\geq0}\varphi(2^{-j}\xi)=1\ \text{for any}\ \xi\in\mathbb{R}^d,
     \\ &\mathrm{supp}\ \varphi(2^{-i}\cdot)\cap\mathrm{supp}\ \varphi(2^{-j}\cdot)=\varnothing,\quad\text{if}\quad |i-j|\geq2,
     \\&\mathrm{supp}\ \varphi(2^{-j}\cdot)\cap\mathrm{supp}\ \chi(x)=\varnothing,\quad\text{if}\quad j\geq1, \\
     &\varphi(\xi)\equiv 1\quad \text{for}\quad\frac43\leq |\xi|\leq \frac32.
      \end{align*}
We can  define the nonhomogeneous localization operators as follows.
\begin{numcases}{\Delta_ju=}
0, &if $j\leq-2$,\nonumber\\
\chi(D)u, &if $j=-1$,\nonumber\\
\varphi(2^{-j}D)u, &if $j\geq0$,\nonumber
\end{numcases}
where the pseudo-differential operator $f(D):u\to\mathcal{F}^{-1}(f \mathcal{F}u)$.

Let us now define the Besov spaces as follows.
\begin{definition}[\cite{B}]
Let $s\in\mathbb{R}$ and $(p,r)\in[1, \infty]^2$. The nonhomogeneous Besov space $B^{s}_{p,r}(\R^d)$ is defined by
\begin{align*}
B^{s}_{p,r}(\R^d):=\f\{f\in \mathcal{S}'(\R^d):\;\|f\|_{B^{s}_{p,r}(\R^d)}<\infty\g\},
\end{align*}
where
\begin{numcases}{\|f\|_{B^{s}_{p,r}(\R^d)}=}
\left(\sum_{j\geq-1}2^{sjr}\|\Delta_jf\|^r_{L^p(\R^d)}\right)^{\fr1r}, &if $1\leq r<\infty$,\nonumber\\
\sup_{j\geq-1}2^{sj}\|\Delta_jf\|_{L^p(\R^d)}, &if $r=\infty$.\nonumber
\end{numcases}
\end{definition}
The following Bernstein's inequalities will be used in the sequel.
\begin{lemma}[\cite{B}] \label{bern} Let $\mathcal{B}$ be a ball and $\mathcal{C}$ be an annulus. There exists a constant $C>0$ such that for all $k\in \mathbb{N}\cup \{0\}$, any $\lambda\in \R^+$ and any function $f\in L^p$ with $1\leq p \leq \infty$, we have
\begin{align*}
&{\rm{supp}}\,\widehat{f}\subset \lambda \mathcal{C}\;\Rightarrow\; C^{-k-1}\lambda^k\|f\|_{L^p} \leq \|\nabla^kf\|_{L^p} \leq C^{k+1}\lambda^k\|f\|_{L^p}.
\end{align*}
\end{lemma}
Finally, we recall some lemmas which will be used in the later of the paper.
\begin{lemma}[\cite{B}]\label{pr}
Let $(p,r)\in[1, \infty]^2$ and $s>0$. For any $f,g \in B^s_{p,r}(\R^d)\cap L^\infty(\R^d)$, we have
\bbal
&\|fg\|_{B^{s}_{p,r}(\R^d)}\leq C\f(\|f\|_{B^{s}_{p,r}(\R^d)}\|g\|_{L^\infty(\R^d)}+\|g\|_{B^{s}_{p,r}(\R^d)}\|f\|_{L^\infty(\R^d)}\g).
\end{align*}
In particular, for $s>\frac{d}p$ or $\{s=\frac{d}p, r=1\}$, we have $B^{s}_{p,r}(\R^d)\hookrightarrow L^\infty(\R^d)$.
\end{lemma}
\begin{lemma}[\cite{B}]\label{cz}
A smooth function $f:\mathbb{R}^{d}\rightarrow \mathbb{R}$ is said to be an $S^{m}$-multiplier: if\:\:$\forall\alpha\in \mathbb{N}^{d}$, there exists a constant $C_{\alpha}>0$ such that
 \begin{align*}
\f|\partial^{\alpha}f(\xi)\g|\leq C_{\alpha}(1+|\xi|)^{m-\alpha},\:\:\xi\in\mathbb{R}^{d}.
    \end{align*}
If $f$ is a $S^{m}$-multiplier, then the operator $f(D)$ is continuous from ${B}_{p,r}^{s}$ to ${B}_{p,r}^{s-m}$ for all $s\in \mathbb{R}$ and $1\leq p,r\leq\infty$.

In particular, there holds that
$$\|(1-\Delta)^{-1}u\|_{B^{s}_{p,r}(\R^d)}\leq C\|u\|_{B^{s-2}_{p,r}(\R^d)}.$$
\end{lemma}
Finally, we recall the regularity estimates for the heat equations.
\begin{lemma}[\cite{B}]\label{reg}
Let $s>0$ and $\epsilon>0$, $1\leq p,r\leq \infty$. Assume that $u_0\in {B}^s_{p,r}$ and $f\in {{L}}^{1}_T{B}^{s}_{p,r}$. Then the heat equations
\begin{align*}
\left\{\begin{array}{ll}
\partial_tu-\epsilon\Delta u=f,\\
 u(t=0,x)=u_0(x),
\end{array}\right.
\end{align*}
has a unique solution $u\in {L}^{\infty}_T{B}^{s}_{p,r}$ satisfying for all $T>0$ and some universal constant $C>0$
\begin{align*}
\|u\|_{{L}^{\infty}_T{B}^{s}_{p,r}}\leq C\f(\|u_0\|_{{B}^s_{p,r}}+\|f\|_{{{L}}^{1}_T{B}^{s}_{p,r}}\g).
\end{align*}
\end{lemma}
\begin{lemma}[\cite{Zhou01,22jde}]\label{wel}
Let $d\geq 1$ and $\ep\geq0$. Assume that $(s,p,r)$ satisfies the conditions \eqref{cond}.
For any initial data $u_0\in B^s_{p,r}(\R^d)$, there exists a finite time $T=T(\|u_0\|_{B^s_{p,r}},\ep)>0$ such that \eqref{1} admits a unique strong solution $u\in\mathcal{C}([0,T_\ep];B^s_{p,r})$.
\end{lemma}

\section{Proof of Theorem \ref{th2}}\label{sec3}
In this section, we will give the proof of Theorem \ref{th2}. Assume $(s,p,r)$ satisfies the conditions \eqref{cond}. For fixed $\epsilon>0$, by Lemma \ref{wel}, we known that there exists a $T_\epsilon=T(\|u_0\|_{B^s_{p,r}},s,\epsilon)>0$ such that the system \eqref{1} has a unique solution $u\in\mathcal{C}([0,T_\epsilon];B^s_{p,r})$. Furthermore, we can obtain that $\exists\; T=T(\|u_0\|_{B^s_{p,r}},s)>0$ such that $T\leq T_{\epsilon}$ and there exists $C_1>0$ independent of $\epsilon$ such that
\begin{align}\label{m1}
\|u\|_{L_T^{\infty} B^s_{p,r}} \leq C_1\left\|u_0\right\|_{B^s_{p,r}}, \quad \forall \epsilon \in[0,1).
\end{align}
Moreover, if $u_0 \in B^\gamma_{p,r} \cap B^s_{p,r}$ for some $\gamma\geq s-1$, then there exists $C_2=C_2(\left\|u_0\right\|_{B^s_{p,r}})>0$ independent of $\epsilon$ such that
\begin{align}\label{m2}
\|u\|_{L_T^{\infty} B^\gamma_{p,r}} \leq C_2\left\|u_0\right\|_{B_{p,r}^\gamma}.
\end{align}

We need to introduce smooth, radial cut-off functions to localize the frequency region. Let $\widehat{\phi}\in \mathcal{C}^\infty_0(\mathbb{R})$ be an even, real-valued and non-negative function on $\R$ and satisfy
\begin{numcases}{\widehat{\phi}(\xi)=}
1, &if\; $|\xi|\leq \frac{1}{4^d}$,\nonumber\\
0, &if\; $|\xi|\geq \frac{1}{2^d}$.\nonumber
\end{numcases}
It is easy to verify that $
\|\phi\|_{L^p(\R)}\approx 1
$ for any $p\in[1,\infty]$.

\begin{lemma}\label{yl1} Assume that $(s,p,r)$ satisfies \eqref{cond}.
Define the initial data $u^n_0$ by
\bbal
&u^n_0(x)=2^{-ns}\phi\left(x_1\right)\cos \left(\frac{17}{12}2^nx_1\right)\phi(x_{2})\cdot\cdot\cdot\phi(x_{d}),\quad n\gg1.
\end{align*}
Then for any $\sigma\in\R$, there exists a positive constant $C=C(\phi)$ such that
\bal
&\|u^n_0\|_{L^\infty}\leq C2^{-ns},\label{0}\\
&C^{-1}2^{n(\sigma-s)}\leq\|u^n_0\|_{B^\sigma_{p,r}}\leq C2^{n(\sigma-s)}.\label{m3}
\end{align}
\end{lemma}
\begin{proof} We just prove \eqref{m3} since \eqref{0} is obvious.
It is easy to show that (cf. \cite[Lemma 3.3]{li20})
\bbal
\mathrm{supp} \ \widehat{u^n_0}&\subset \left\{\xi\in\R^d: \ \frac{17}{12}2^n-\fr12\leq |\xi|\leq \frac{17}{12}2^n+\fr12\right\}
\subset \left\{\xi\in\R^d: \ \frac{4}{3}2^n\leq |\xi|\leq \frac{3}{2}2^n\right\}.
\end{align*}
Due to the fact $\varphi(\xi)\equiv 1$ for $\frac43\leq |\xi|\leq \frac32$, we have
\begin{numcases}{\Delta_j(u^n_0)=}
u^n_0, &if $j=n$,\nonumber\\
0, &if $j\neq n$.\nonumber
\end{numcases}
Thus, we deduce that
\bbal
\|u^n_0\|_{B^\sigma_{p,r}(\R^d)}&
= 2^{n(\sigma-s)}\f\|\phi\left(x_1\right)\cos \left(\frac{17}{12}2^nx_1\right)\g\|_{L^{p}(\R)}\|\phi\|^{d-1}_{L^{p}(\R)}
\approx2^{n(\sigma-s)},
\end{align*}
This completes the proof of Lemma \ref{yl1}.
\end{proof}

{\bf Proof of Theorem \ref{th2}.}\; From now on, we set $\epsilon:=2^{-2n}$.
We decompose the solution $u^{\ep}$ to \eqref{1} into three parts
\begin{align}\label{U1}
u^{\epsilon}=u_1^{\epsilon}+u_2^{\epsilon}+ u_3^{\epsilon},
\end{align}
where $u^{\epsilon}_1$ and $u^{\epsilon}_2$ solve the following system, respectively,
\begin{align*}
\begin{cases}
\partial_{t}u_1^{\epsilon}-\epsilon\Delta u_1^{\epsilon}=0,\\
u_1^{\epsilon}(t=0,x)=u^n_0(x),
\end{cases}
\end{align*}
and
\begin{align}\label{lu2}
\begin{cases}
\partial_{t}u_2^{\epsilon}-\epsilon\Delta u_2^{\epsilon}=-\div\f(u_1^{\epsilon}(1-u_1^{\epsilon})\nabla S_1^{\epsilon}\g),\\
S_1^{\epsilon}=(1-\Delta)^{-1}u_1^{\epsilon},        \\
u_2^{\epsilon}(t=0,x)=0.
\end{cases}
\end{align}
We decompose the solution $\bar{u}$ to \eqref{2} into three parts
\begin{align}\label{U2}
\bar{u}=\bar{u}_1+\bar{u}_2+ \bar{u}_3,
\end{align}
where $\bar{u}_1=u_0^n$ and $\bar{u}_2$ solves the following system
\begin{align}\label{lu23}
\begin{cases}
\partial_{t}\bar{u}_2=-\div\f(\bar{u}_1(1-\bar{u}_1)\nabla \bar{S}_1\g),\\
\bar{S}_1=(1-\Delta)^{-1}\bar{u}_1,        \\
\bar{u}_2(t=0,x)=0.
\end{cases}
\end{align}
It follows from \eqref{U1} and \eqref{U2} that
\bbal
u^{\epsilon}(t, u^n_0)-u^{0}(t, u^n_0)&=\underbrace{u_1^{\epsilon}(t, u^n_0)-\bar{u}_1(t, u^n_0)}_{=:\,\mathrm{I}_1}+\underbrace{u_2^{\epsilon}(t, u^n_0)-\bar{u}_2(t, u^n_0)}_{=:\,\mathrm{I}_2}+\underbrace{u_3^{\epsilon}(t, u^n_0)-\bar{u}_3(t, u^n_0)}_{=:\,\mathrm{I}_3}.
\end{align*}
{\bf Claim:} For small time $t\in(0,T_0]$, there hold that
\begin{enumerate}
  \item $\f\|\mathrm{I}_1\g\|_{B^s_{p,r}}\thickapprox t$,
  \item $\f\|\mathrm{I}_2\g\|_{B^s_{p,r}}\leq \varepsilon_n\to0, \;\text{as}\; n\to\infty$,
  \item $\f\|\mathrm{I}_3\g\|_{B^s_{p,r}}\leq Ct^2$.
\end{enumerate}
Assuming that the {\bf Claim} holds, using the reverse triangle inequality, we deduce that for small time $t\in(0,T_0]$
\bbal
\liminf_{n\to \infty}\f\|u^{\epsilon}(t, u^n_0)-u^{0}(t, u^n_0)\g\|_{B^s_{p,r}}\geq c_0t>0.
\end{align*}
This completes the proof of Theorem \ref{th2}.

It remains to prove the above {\bf Claim}. Next we prove the {\bf Claim} by presenting the following three Lemmas.
For the sake of convenience, we write $u_i^{\epsilon}:=u_i^{\epsilon}(t, u^n_0)$ and $\bar{u}_i:=\bar{u}_i(t, u^n_0)$ with $i=1,2,3$.
\begin{lemma}[Estimation of $\mathrm{I}_1$]\label{pr1} Let $u^n_0$ be given by Lemma \ref{yl1}. Assume that $(s,p,r)$ satisfies \eqref{cond}.
Then there exists two positive constants $C_1$ and $C_2$ such that for $t\in (0,1)$
\begin{align*}
C_1 t\leq\left\|(u_1^{\epsilon}-\bar{u}_1)(t)\right\|_{B^{s}_{p,r}}\leq C_2 t.
\end{align*}
\end{lemma}
\begin{proof}
By the mean-value Theorem, one has
\bbal
u_1^{\epsilon}-\bar{u}_1=\f(e^{t\epsilon \Delta}-\mathrm{Id}\right)u_0^n=\int_0^t\f(e^{\tau\epsilon \Delta}\epsilon \Delta u_0^n\g) \dd\tau .
\end{align*}
Notice that $\mathrm{supp} \ \widehat{u^n_0}\subset\left\{\xi\in\R^d: |\xi|\sim \epsilon^{-\fr12}\g\}$,
from Lemmas \ref{reg} and \ref{bern}, we obtain that
\bbal
\left\|\f(e^{t\epsilon \Delta}-\mathrm{Id}\right)u_0^n\right\|_{B^{s}_{p,r}}\leq\int_0^t\left\|e^{\tau\epsilon \Delta}\epsilon \Delta u_0^n\right\|_{B^{s}_{p,r}}\dd\tau\leq Ct\left\|\epsilon \Delta u_0^n\right\|_{B^{s}_{p,r}}\approx t,
\end{align*}
which gives that
\bbal
\left\|u_1^{\epsilon}-\bar{u}_1\right\|_{B^{s}_{p,r}}\geq t\f\|\epsilon \Delta u_0^n\g\|_{B^{s}_{p,r}}-\int_0^t\left\|\f(e^{\tau\epsilon \Delta}-\mathrm{Id}\right)\epsilon \Delta u_0^n\right\|_{B^{s}_{p,r}}\dd\tau\geq C_1t.
\end{align*}
This completes the proof of Lemma \ref{pr1}.
\end{proof}
\begin{lemma}[Estimation of $\mathrm{I}_2$]\label{pr2} Let $u^n_0$ be given by Lemma \ref{yl1}. Assume that $(s,p,r)$ satisfies \eqref{cond}.
Then there exists some positive constant $C$ such that for $t\in (0,1)$
\begin{align*}
\left\|\f(u_2^{\epsilon},\bar{u}_2\g)(t)\right\|_{B^{s}_{p,r}}\leq C2^{-n(s-1)}.
\end{align*}
\end{lemma}
\begin{proof}
From \eqref{lu23}, one has
\begin{align}\label{3}
&\bar{u}_2(t)=-\int_0^t\div\f(\bar{u}_1(1-\bar{u}_1)\nabla \bar{S}_1\g)\dd\tau.
\end{align}
Using Lemmas \ref{pr}-\ref{cz} and Lemma \ref{yl1}, we have
\bbal
\left\|\div\f(\bar{u}_1(1-\bar{u}_1)\nabla \bar{S}_1\g)\right\|_{B^{s}_{p,r}}
&\leq C\left\|\bar{u}_1(1-\bar{u}_1)\nabla \bar{S}_1\right\|_{B^{s+1}_{p,r}}\\
&\leq C\left\|\bar{u}_1(1-\bar{u}_1)\right\|_{L^{\infty}}\|\nabla \bar{S}_1\|_{B^{s+1}_{p,r}}
+C\left\|\bar{u}_1(1-\bar{u}_1) \right\|_{B^{s+1}_{p,r}}\|\nabla \bar{S}_1\|_{L^{\infty}}\\
&\leq C\f(1+\left\|\bar{u}_1\right\|_{L^{\infty}}\g)\left\|\bar{u}_1\right\|_{L^{\infty}}\|\bar{u}_1\|_{B^{s}_{p,r}}
+C\f(1+\left\|\bar{u}_1\right\|_{L^{\infty}}\g)\left\|\bar{u}_1\right\|_{L^{\infty}}\|\bar{u}_1\|_{B^{s+1}_{p,r}}\\
&\leq C\f(1+\|u_0^n\|_{L^{\infty}}\g)\|u_0^n\|_{L^\infty}\|u^n_0\|_{B^{s+1}_{p,r}}\\
&\leq C2^{-n(s-1)},
\end{align*}
where we also have used $\|\nabla \bar{S}_1\|_{L^{\infty}}\leq C\|\bar{u}_1\|_{L^{\infty}}$.

Inserting the above estimate into \eqref{3} yields that
$$\left\|\bar{u}_2(t)\right\|_{B^{s}_{p,r}}\leq C2^{-n(s-1)}t.$$
The estimation of $u_2^{\epsilon}$ can be similarly done by using the regularity estimate of the heat equation with zero initial data.
This completes the proof of Lemma \ref{pr2}.
\end{proof}

\begin{lemma}[Estimation of $\mathrm{I}_3$]\label{pr3} Let $u^n_0$ be given by Lemma \ref{yl1}. Assume that $(s,p,r)$ satisfies \eqref{cond}.
Then there exists some positive constants $C$ such that for $t\in (0,1)$
\begin{align*}
\left\|\f(u_3^{\epsilon},\bar{u}_3\g)(t)\right\|_{B^{s}_{p,r}}\leq C t^2.
\end{align*}
\end{lemma}
\begin{proof} By Duhamel's principle, one has
\begin{align}\label{l}
&u_1^{\epsilon}(t)=e^{t\epsilon\Delta}u^n_0,\\
&u^{\epsilon}(t)=u_1^{\epsilon}(t)-\int_0^te^{(t-\tau)\epsilon\Delta}\div\f(u^{\epsilon}(1-u^{\epsilon})\nabla S^{\epsilon}\g)\dd\tau,
\end{align}
which gives that
\bal\label{L1}
\left\|u^{\epsilon}-u_1^{\epsilon}\right\|_{B^{s-1}_{p,r}}
\leq&~ C\int_0^t\left\|\div\f(u^{\epsilon}(1-u^{\epsilon})\nabla S^{\epsilon}\g)\right\|_{B^{s-1}_{p,r}}\dd\tau\nonumber\\
\leq&~ C\int_0^t\left\|u^{\epsilon}(1-u^{\epsilon})\nabla S^{\epsilon}\right\|_{B^{s}_{p,r}}\dd\tau\nonumber\\
\leq&~ C\int_0^t\f(\left\|u^{\epsilon}(1-u^{\epsilon})\right\|_{B^{s}_{p,r}} \left\|\nabla S^{\epsilon}\right\|_{L^{\infty}}+\left\|u^{\epsilon}(1-u^{\epsilon})\right\|_{L^{\infty}} \left\|\nabla S^{\epsilon}\right\|_{B^{s}_{p,r}}\g)\dd\tau\nonumber\\
\leq&~ C\int_0^t\f(1+\|u^{\epsilon}\|_{L^\infty}\g)\|u^{\epsilon}\|_{B^{s}_{p,r}} \left\|u^{\epsilon}\right\|_{B^{s-1}_{p,r}}\dd\tau\nonumber\\
\leq&~ C\int_0^t\f(1+\|u_0^n\|_{B^{s-1}_{p,r}}\g)\|u_0^n\|_{B^{s}_{p,r}} \|u_0^n\|_{B^{s-1}_{p,r}}\dd\tau\nonumber\\
\leq&~ Ct2^{-n}.
\end{align}
Similarly, we also have
\bal
\left\|u^{\epsilon}-u_1^{\epsilon}\right\|_{B^{s}_{p,r}}
\leq&~ C\int_0^t\left\|u^{\epsilon}(1-u^{\epsilon})\right\|_{B^{s+1}_{p,r}} \left\|\nabla S^{\epsilon}\right\|_{B^{s-1}_{p,r}}+\left\|u^{\epsilon}(1-u^{\epsilon})\right\|_{B^{s-1}_{p,r}} \left\|\nabla S^{\epsilon}\right\|_{B^{s+1}_{p,r}}\dd\tau
\leq Ct,\label{L2}\\
\left\|u^{\epsilon}-u_1^{\epsilon}\right\|_{B^{s+1}_{p,r}}
\leq&~ C\int_0^t\left\|u^{\epsilon}(1-u^{\epsilon})\right\|_{B^{s+2}_{p,r}} \left\|\nabla S^{\epsilon}\right\|_{B^{s-1}_{p,r}}+\left\|u^{\epsilon}(1-u^{\epsilon})\right\|_{B^{s-1}_{p,r}} \left\|\nabla S^{\epsilon}\right\|_{B^{s+2}_{p,r}}\dd\tau
\leq Ct2^{n}.\label{L3}
\end{align}
From \eqref{l}, we have
\begin{align*}
\begin{cases}
\partial_{t}u_3^{\epsilon}-\epsilon\Delta u_3^{\epsilon}=-\div\f(u^{\epsilon}(1-u^{\epsilon})\nabla S^{\epsilon}-u_1^{\epsilon}(1-u_1^{\epsilon})\nabla S_1^{\epsilon}\g),\\
S_1^{\epsilon}=(1-\Delta)^{-1}u_1^{\epsilon},        \\
u_3^{\epsilon}(t=0)=0.
\end{cases}
\end{align*}
Using Lemma \ref{reg}, we obtain that
\bal\label{yy1}
\|u_3^{\epsilon}\|_{B^{s}_{p,r}}
\leq&~ C\int_0^t\left\|u^\epsilon(1-u^\epsilon)\nabla S^\epsilon-u_1^{\epsilon}(1-u_1^{\epsilon})\nabla S_1^{\epsilon}\right\|_{B^{s+1}_{p,r}}\dd\tau\leq C \int_0^t \sum_{i=1}^4\mathrm{J}_i(\tau) \dd\tau,
\end{align}
where
\bbal
&\mathrm{J}_1:=\left\|u^\epsilon(1-u^\epsilon)-u_1^{\epsilon}(1-u_1^{\epsilon})\right\|_{L^\infty} \f\|\nabla S^\epsilon\g\|_{B^{s+1}_{p,r}},\\
&\mathrm{J}_2:=\left\|u^\epsilon(1-u^\epsilon)-u_1^{\epsilon}(1-u_1^{\epsilon})\right\|_{B^{s+1}_{p,r}} \f\|\nabla S^\epsilon\g\|_{L^\infty},\\
&\mathrm{J}_3:=\left\|u_1^{\epsilon}(1-u_1^{\epsilon})\right\|_{L^\infty}
\f\|\nabla( S^\epsilon-S_1^\epsilon)\g\|_{B^{s+1}_{p,r}},\\
&\mathrm{J}_4:=\left\|u_1^{\epsilon}(1-u_1^{\epsilon})\right\|_{B^{s+1}_{p,r}} \f\|\nabla( S^\epsilon-S_1^\epsilon)\g\|_{L^\infty}.
\end{align*}
Next, we need to estimate the above four terms one by one.

Using Lemmas \ref{pr}-\ref{cz} and Lemma \ref{yl1}, one has
\bbal
\mathrm{J}_1(t)\leq&~ C\left\|\f(u^\epsilon-u_1^{\epsilon}\g)\f(1-u^\epsilon-u_1^{\epsilon}\g)\right\|_{B^{s-1}_{p,r}}
\|u^\epsilon\|_{B^{s}_{p,r}}\\
\leq&~ C\left\|\f(u^\epsilon-u_1^{\epsilon}\g)(t)\right\|_{B^{s-1}_{p,r}}\f(1+\|u^\epsilon\|_{B^{s-1}_{p,r}}+\| u_1^{\epsilon}\|_{B^{s-1}_{p,r}}\g)
\|u^\epsilon\|_{B^{s}_{p,r}}\\
\leq&~ C\left\|\f(u^\epsilon-u_1^{\epsilon}\g)(t)\right\|_{B^{s-1}_{p,r}}\f(1+\|u^n_0\|_{B^{s-1}_{p,r}}\g)
\|u^n_0\|_{B^{s}_{p,r}}\\
\leq&~ C2^{-n}t,\\
\mathrm{J}_2(t)\leq&~C\left\|\f(u^\epsilon-u_1^{\epsilon}\g)\f(1-u^\epsilon-u_1^{\epsilon}\g)\right\|_{B^{s+1}_{p,r}}
\|u^\epsilon\|_{B^{s-1}_{p,r}}\\
\leq&~C\left\|\f(u^\epsilon-u_1^{\epsilon}\g)\right\|_{B^{s+1}_{p,r}}\f(1+\|u^n_0\|_{B^{s-1}_{p,r}}\g)
\|u^n_0\|_{B^{s-1}_{p,r}}\\
\quad&+C\left\|\f(u^\epsilon-u_1^{\epsilon}\g)\right\|_{B^{s-1}_{p,r}}\f(1+\|u^n_0\|_{B^{s+1}_{p,r}}\g)
\|u^n_0\|_{B^{s-1}_{p,r}}\\
\leq&~Ct,\\
\mathrm{J}_3(t)\leq&~ C\left\|u_1^{\epsilon}(1-u_1^{\epsilon})\right\|_{B^{s-1}_{p,r}}
\|u^\epsilon-u_1^\epsilon\|_{B^{s}_{p,r}}\\
\leq&~ C\f(1+\|u^n_0\|_{B^{s-1}_{p,r}}\g)\| u_0^n\|_{B^{s-1}_{p,r}}\|u^\epsilon-u_1^\epsilon\|_{B^{s}_{p,r}}\\
\leq&~C2^{-n}t,\\
\mathrm{J}_4(t)\leq&~C\left\|u_1^{\epsilon}(1-u_1^{\epsilon})\right\|_{B^{s+1}_{p,r}}
\|u^\epsilon-u_1^\epsilon\|_{B^{s-1}_{p,r}}\\
\leq&~ C\f(1+\|u^n_0\|_{B^{s-1}_{p,r}}\g)\| u_0^n\|_{B^{s+1}_{p,r}}\|u^\epsilon-u_1^\epsilon\|_{B^{s-1}_{p,r}}\\
\leq&~ Ct,
\end{align*}
where we have used \eqref{L1}-\eqref{L3}. Inserting the above estimations into \eqref{yy1}, then we obtain
$$\|u_3^{\epsilon}\|_{B^{s}_{p,r}}\leq Ct^2.$$
Notice that
\bbal
\bar{u}_3(t)
=&~ -\int_0^t\div\f(\bar{u}(1-\bar{u})\nabla \bar{S}-\bar{u}_1(1-\bar{u}_1)\nabla \bar{S}_1\g)\dd\tau.
\end{align*}
and repeating the above procedure, we can complete the proof of Lemma \ref{pr3}.
\end{proof}

\section*{Acknowledgements}
Y. Yu is supported by the National Natural Science Foundation of China (12101011).

\section*{Declarations}
\noindent\textbf{Data Availability} No data was used for the research described in the article.

\vspace*{1em}
\noindent\textbf{Conflict of interest}
The authors declare that they have no conflict of interest.

\end{document}